\documentclass[12pt,reqno]{amsart}
\usepackage{amssymb,amsmath,mathtools,amsthm,esint}
\usepackage{hyperref}
\usepackage[active]{srcltx}
\usepackage{fullpage}
\usepackage[T1]{fontenc}
\usepackage{xpatch}
\usepackage{algorithm}
\usepackage[n,advantage,operators,sets,adversary,landau,
probability,notions,logic,ff,mm,primitives,events,complexity,asymptotics,keys]{cryptocode}
\usepackage[noend]{algpseudocode}
\usepackage{microtype}
\makeatletter
\def\BState{\State\hskip-\ALG@thistlm}
\makeatother
\xpatchcmd{\proof}{.}{\proofpunctuation}{}{}
\newcommand{\proofpunctuation}{:}
\title[Big $O$-Estimate]
{Sharp upper bound for the sixth moment of the Riemann zeta function}
\author[Thi Altenschmidt]{Thi Altenschmidt}
\email{ldt.altenschmidt@gmail.com}
\thanks{Ich m\"ochte mich an dieser Stelle bei meiner Ehefrau Sabine Liesbeth Altenschmidt f\"ur ihre unendliche Liebe, Geduld und Unterst\"utzung bedanken.}
\begin{document}
\theoremstyle{plain}
\newtheorem{thm}{Theorem}
\newtheorem{theorem}[thm]{Theorem}
\newtheorem{cor}[thm]{Corollary}
\theoremstyle{definition}
\numberwithin{equation}{thm}
\begin{abstract}
The main task of this work is to give an improvement for the upper bounds of the Laplace transform $\forall \beta \geqslant 2$ 
\[\int_0^{+\infty}\abs{\zeta\left(\frac{1}{2}+it\right)}^{2\beta}e^{-\delta t}dt \ll_{\beta,\varepsilon}  \max\left\{\frac{1}{\delta^{\frac{\beta-1}{2}+\varepsilon}},\frac{1}{\delta^{1+\varepsilon}},\frac{1}{\delta^{\beta-2+\varepsilon}} \right\}\]
for
\[\quad 0 < \delta < \frac{\pi}{2}, \delta \to 0^+, \forall \varepsilon > 0.\] 
In particular, this implies 
\[\int_0^T \abs{\zeta\left(\frac{1}{2}+it\right)}^{6}dt \ll_{\beta,\varepsilon} T^{1+\varepsilon} \quad T \to +\infty, \forall \varepsilon > 0.\]
\end{abstract} 
\maketitle
\section*{Introduction}
Let $\beta \in \mathbb{N}$ be a natural number. The $2\beta$-th integer absolute moment of the Riemann zeta function on the critical line is defined as 
\[M_\beta(T) = \int_0^T \abs{\zeta \left ( \frac{1}{2}+it \right ) }^{2\beta}dt.\] 
The main conjecture in estimation of the absolute moments is the following asymptotic formula
\[M_\beta(T) \sim C_\beta \cdot T \cdot \log^{\beta^2}(T), \quad T \to +\infty,\]
for some positive constants $C_\beta$ (\cite[p. 65]{HB81}). In \cite{KS00} the conjecture is stated with explicite coefficients $C_\beta$. For $\beta=0$ the formula is trivial. For $\beta=1,2$ it is well-known that $C_1 = 1$, resp.  $C_2 = \frac{1}{2\pi^2}$ from the classical works of G.H. Hardy und J.E. Littlewood \cite{HaLi18} resp. A.E. Ingham \cite{Ing26}. In fact, if one could prove only the sharp upper bounds $M_\beta(T) \ll_{\beta,\varepsilon} T^{1+\varepsilon}$ for all $\beta \geqslant 2$, then this already implies the Lindel\"of hypothesis, which says $\zeta\left(\frac{1}{2}+it\right) \ll_\varepsilon \abs{t}^{\varepsilon}$, because (cf. \cite[p. 45]{Iv85})
\[\abs{\zeta(1/2+it)}^{2\beta} \ll \log T \left(1 + \int_{T-\log^2 T}^{T+\log^2 T}\abs{\zeta\left(\frac{1}{2}+it \right)}^{2\beta}dt \right).\]
On Riemann hypothesis K. Soundararajan proved in \cite{S09} that $M_\beta(T) \ll_{\beta,\varepsilon} T\left( \log T\right)^{\beta^2+\varepsilon}$ for every positive real number $\beta$ and $\forall \varepsilon > 0$ and this has been improved in \cite{Ha13}. In \cite[Chap. VII]{Tit86} one considers the Laplace transform
\[J_\beta(\delta) = \int_0^{+\infty}\abs{\zeta\left(\frac{1}{2}+it\right)}^{2\beta}e^{-\delta t}dt, \quad 0 < \delta < \frac{\pi}{2}\]
and the classical result is (cf. \cite[Chap. VII, p. 173]{Tit86})
\[J_\beta(2\delta) = O_{\beta,\varepsilon}\left(\frac{1}{\delta^{\frac{\beta}{2}+\varepsilon}}\right), \quad \delta \to 0^+, \forall \varepsilon > 0, \beta \geqslant 2.\]
In this work we improve this classical result by proving the following
\begin{theorem}\label{thm1}
As $\delta \to 0^+$
\[J_\beta(\delta) \ll_{\beta,\varepsilon} \max \left\{  \frac{1}{\delta^{\frac{\beta-1}{2}+\varepsilon}},\frac{1}{\delta^{1+\varepsilon}}, \frac{1}{\delta^{\beta-2+\varepsilon}} \right\}, \quad \forall \varepsilon > 0, \forall \beta \geqslant 2.\]
\end{theorem}
\noindent Consequently we obtain
\begin{cor}
As $T \to +\infty$
\[M_3(T) \ll_{\varepsilon} T^{1+\varepsilon}, \quad T \to +\infty, \forall \varepsilon > 0.\]
\end{cor}
\begin{proof}
Let $f(t) \geqslant 0$ be a non-negative function. For $\delta > 0$, if we have $\forall \varepsilon > 0$
\[\int_0^{+\infty}f(t)e^{-\delta t}dt \ll_{\varepsilon} g(1/\delta,\varepsilon), \quad \delta \to 0^+\]
for some function $g$
then 
\[\int_0^Tf(t)dt \ll_{\varepsilon} g(T,\varepsilon), \quad T \to +\infty, \forall \varepsilon > 0.\]
Indeed, we have 
\[\int_0^Tf(t)dt < e \cdot \int_0^T f(t)e^{-t/T}dt < e \cdot \int_0^{+\infty}f(t)e^{-t/T}dt \ll_{\varepsilon} g(T,\varepsilon).\]
So the corollary follows immediately from theorem \ref{thm1}.
\end{proof}
\noindent We will follow the main framework \cite[Chap. VII \S 7.13 - \S 7.18]{Tit86} and include all the details of our own arguments. We explain briefly now our idea. In \cite[Chap. VII \S 7.18]{Tit86} one breaks down
\[\int_0^{+\infty}\left|\Phi_\beta\left(xie^{-i\delta}\right) \right|^2dx = \int_0^{\lambda}\left|\Phi_\beta\left(xie^{-i\delta}\right) \right|^2dx + \int_{\lambda}^{+\infty}\left|\Phi_\beta\left(xie^{-i\delta}\right) \right|^2dx,\]
where
\[\Phi_\beta(z) = \sum_{n=1}^{+\infty}d_\beta(n)e^{-nz} + \mathsf{Res}_{s=1}\frac{\Gamma(s)\zeta^\beta(s)}{z^s}\]
then integrates part by part, which leads to the problem with the estimation of the double sums over all natural numbers $n,m$ with $n \neq m$ involving $d_\beta(n)d_\beta(m)$, since one can choose only $\lambda = \delta^{\frac{\beta}{2}}$ as the optimal choice to balance the upper bounds of both integrals above. Instead of that, we break down
\begin{multline*}
\int_0^{+\infty}\left|\Phi_\beta\left(xie^{-i\delta}\right) \right|^2dx \\ = \int_1^{+\infty}\left|\Phi_\beta\left(xie^{-i\delta}\right) \right|^2dx + \int_0^{\delta^{\frac{\beta-1}{2}}}\left|\Phi_\beta\left(xie^{-i\delta}\right) \right|^2dx + \int_{\delta^{\frac{\beta-1}{2}}}^1\left|\Phi_\beta\left(xie^{-i\delta}\right) \right|^2dx.
\end{multline*}
The upper bounds for the first two integrals are well-known. The tricky part is now to estimate the third integral. Prof. B. Conrey kindly pointed out to me by a simple argument that one can't get rid off the oscillations in the sum $\sum_{n=1}^{+\infty}d_\beta(n)e^{-nz}$. Let now $\alpha \in (0,1)$ be a real number. By method of moving contour integrals and residue calculation we shift firstly to the negative vertical line
\[\int_{-\alpha-i\infty}^{-\alpha+i\infty}\Gamma(s)\zeta^\beta(s)\frac{ds}{\left(xie^{-i\delta}\right)^s}\]
then use the functional equation $\zeta(s) = \chi(s)\zeta(1-s)$ to turn back to the line $\sigma = 1+\alpha$. This is the reflection principle (see e.g \cite[\S 4.4]{Iv85}). The crucial point is that the two integrals
\[\int_{-\alpha-i\infty}^{-\alpha+i\infty}\Gamma(s)\zeta^\beta(s)\frac{ds}{\left(xie^{-i\delta}\right)^s}, \quad \int_{-\alpha-i\infty}^{-\alpha+i\infty}\left|\frac{\Gamma(s)\zeta^\beta(s)}{\left(xie^{-i\delta}\right)^s}\right|\left|ds\right|\]
are quite different. The oscillations in the first one will cancel out each other, while the second one is very huge. With the techniques of exponential integrals (see \cite[\S 2.1]{Iv85}) we can reduce the higher power of $\frac{1}{\delta}$ in the oscillatory integral by $1/2$. This old idea has been exploited already very successfully in \cite[Chap. XII, p. 315]{Tit86} for estimating the rest $\Delta_\beta(x)$ of the general divisor sum $D_\beta(x) = \sum_{n \leqslant x}d_\beta(n)$ subtracted by $xP_{\beta-1}(x)$. Now, by choosing $\alpha = \frac{\varepsilon}{2(\beta-1)}$, we get the desired estimation. Since we must rely on the classical result in \cite[Chap. VII, p.173]{Tit86}
\[\int_0^{\delta^{\frac{\beta-1}{2}}}\left|\Phi_\beta(xie^{-i\delta}\right|^2dx = O_{\beta,\varepsilon}\left(\frac{1}{\delta^{\frac{\beta-1}{2}+\varepsilon}}\right), \quad \forall \beta \geqslant 2, \forall \varepsilon > 0, \delta \to 0^+,\]
it is clear that our method won't work in higher cases $\beta \geqslant 4$. Now we fix some notations throughout this work. We will denote by $C_\beta$ some positive constants, which depend only on $\beta$ and they must not be the same everywhere. By $P_{\beta-1}(z)$ we will mean some complex polynomials in $z$ of degree $\beta-1$ and again they can be different from each other. All the asymptotic notations will be written freely either in Landau's or in Vinogradov's style. 
\section{Proof of theorem \ref{thm1}}
\noindent By analytic continuation of the Cahen-Mellin's formula one has for a complex number $z$ with real part $\mathfrak{R}(z) > 0$ and argument $\left|\mathrm{Arg}(z)\right| < \frac{\pi}{2}$
\[\frac{1}{2\pi i}\int_{2-i\infty}^{2+i\infty}\Gamma(s)\zeta^\beta(s)\frac{ds}{z^s} = \sum_{n=1}^{+\infty}d_\beta(n) \frac{1}{2\pi i}\int_{2-i\infty}^{2+i\infty}\Gamma(s)\frac{ds}{(nz)^s} = \sum_{n=0}^{+\infty}d_\beta(n)e^{-nz},\]
where 
\[d_\beta(n) = \sum_{n = n_1 n_2 \dots n_\beta}1.\]
The function $\Gamma(s)\zeta^\beta(s) \frac{1}{z^s}$ has a pole at $s=1$. We define
\[\Psi_\beta(z) = \mathsf{Res}_{s=1} \left ( \Gamma(s)\zeta^\beta(s)\frac{1}{z^s} \right )\]
and
\[\Phi_\beta(z) = \sum_{n=0}^{+\infty}d_\beta(n)e^{-nz} - \Psi_\beta(z).\]
Residue calculation gives us 
\[\Psi_\beta(z) = \frac{1}{z}\left \{ \lambda_0 - \lambda_1\log(z) + \dots + (-1)^{\beta-1}\lambda_{\beta-1}\frac{\log^{\beta-1}(z)}{(\beta-1)!} \right \},\]
where $\lambda_0,\dots,\lambda_{\beta-1}$ are the coefficients in the series representation near the pole $s=1$ of
\[\Gamma(s)\zeta^\beta(s) = \sum_{n=0}^{\beta-1} \frac{\lambda_n}{(s-1)^{n+1}} + \sum_{n=0}^{+\infty}a_n(s-1)^n.\]
Indeed, one can write
\[\frac{1}{z^s} = \frac{1}{z}\exp \left(-(s-1)\log z\right) = \frac{1}{z}\sum_{n=0}^{+\infty}\frac{(-1)^n}{n!}(s-1)^n \log^n z,\]
from which the residue calculation follows easily. Cauchy's integral theorem along the segments $[2-iT,2+iT,\frac{1}{2}+iT,\frac{1}{2}-iT]$ gives us 
\begin{multline*}
2\pi i \Psi_\beta(z) = \\
\int_{2-iT}^{2+iT}\Gamma(s)\zeta^\beta(s) \frac{ds}{z^s} + \int_{2+iT}^{\frac{1}{2}+iT}\Gamma(s)\zeta^\beta(s) \frac{ds}{z^s}  + \int_{\frac{1}{2}+iT}^{\frac{1}{2}-iT}\Gamma(s)\zeta^\beta(s) \frac{ds}{z^s}+ \int_{\frac{1}{2}-iT}^{2-iT}\Gamma(s)\zeta^\beta(s) \frac{ds}{z^s}. 
\end{multline*}
One has 
\[\int_{2+iT}^{\frac{1}{2}+iT}\Gamma(s)\zeta^\beta(s) \frac{ds}{z^s} = \int_2^{\frac{1}{2}}\Gamma(\sigma+iT)\zeta^\beta(\sigma+iT) \frac{d\sigma}{z^{\sigma+iT}},\]
where $z^{\sigma+iT} = \exp((\sigma+iT)\log(z))$ and $\log z$ takes the principal branch. By using the Stirling's formula
\[\Gamma(s) = \sqrt{2\pi}\abs{t}^{\sigma-1/2}e^{-\pi\abs{t}/2}\left (1 + O \left (\abs{t}^{-1} \right ) \right ), \quad \abs{t} \geqslant T_0 \]
and the trivial estimation (\cite[Thm. 1.9]{Iv85})
\[ \zeta(\sigma+it) \ll \begin{cases} \log t, \quad 1 \leqslant \sigma \leqslant 2 \\ t^{\frac{ 1-\sigma} {2}}\log t, \quad 0 \leqslant \sigma \leqslant 1 \end{cases} \]
uniformly in $\sigma$ for $t \geqslant T_0^{\prime}$, we see that
\[\lim_{T \to +\infty}\int_{2+iT}^{\frac{1}{2}+iT}\Gamma(s)\zeta^\beta(s) \frac{ds}{z^s} = 0.\]
By using $\zeta(s) = \overline{\zeta(\bar{s})}$ we also see that
\[\lim_{T \to +\infty}\int_{\frac{1}{2}-iT}^{2-iT}\Gamma(s)\zeta^\beta(s) \frac{ds}{z^s} = 0.\]
Hence
\[\Phi_\beta(z) = \int_{\frac{1}{2}-i\infty}^{\frac{1}{2}+i\infty}\Gamma(s)\zeta^\beta(s)\frac{ds}{z^s}.\]
Now we set $z = xe^{i\left (\frac{\pi}{2} - \delta \right )}$ for $x > 0$ and $0 < \delta < \frac{\pi}{2}$. Parseval's theorem tells us
\[\frac{1}{2\pi}\int_{-\infty}^{+\infty}\abs{\Gamma \left ( \frac{1}{2}+it \right )\zeta^\beta \left (\frac{1}{2}+it \right )}^2e^{(\pi-2\delta)t}dt = \int_0^{+\infty}\abs{\Phi_\beta \left (xe^{i\left( \frac{\pi}{2}-\delta \right)} \right)}^2dx.\]
By using the Stirling's formula and the trivial estimation of the Riemann zeta function as above, there must exist $t_0 = \max(T_0,T_0^{\prime}) > 0$, such that
\begin{multline*}
\int_{-\infty}^0\abs{\Gamma \left ( \frac{1}{2}+it \right )\zeta^\beta \left (\frac{1}{2}+it \right )}^2e^{(\pi-2\delta)t}dt = \int_{-\infty}^{-t_0}\abs{\Gamma \left ( \frac{1}{2}+it \right )\zeta^\beta \left (\frac{1}{2}+it \right )}^2e^{(\pi-2\delta)t}dt \\ + \int_{-t_0}^0\abs{\Gamma \left ( \frac{1}{2}+it \right )\zeta^\beta \left (\frac{1}{2}+it \right )}^2e^{(\pi-2\delta)t}dt \\ = O_{\beta,\varepsilon}(1) +   O \left ( \int_{-\infty}^{-t_0}e^{2(\pi-\delta)t}\abs{t}^{\frac{\beta}{2}} \log^{2\beta} \abs{t} \left ( 1 + O \left (\abs{t}^{-1} \right ) \right ) dt \right ) = O_{\beta,\varepsilon}(1).
\end{multline*}
So we must have
\[J_\beta(2\delta) = \int_0^{+\infty}\abs{\zeta \left (\frac{1}{2} + it \right )}^{2\beta}e^{-2\delta t}dt = \int_0^{+\infty}\abs{\Phi_\beta \left ( xe^{i\left ( \frac{\pi}{2} - \delta \right )} \right )}^2dx + O_{\beta,\varepsilon}(1).\]
We decompose
\[\int_0^{+\infty}\abs{\Phi_\beta \left ( xe^{i\left ( \frac{\pi}{2} - \delta \right )} \right )}^2dx = \int_1^{+\infty}\abs{\Phi_\beta \left ( xe^{i\left ( \frac{\pi}{2} - \delta \right )} \right )}^2dx + \int_0^1\abs{\Phi_\beta \left ( xe^{i\left ( \frac{\pi}{2} - \delta \right )} \right )}^2dx = J_1 + J_2.\]
The estimation of $J_1$ is well-known to the experts, but we include the arguments here for the sake of completeness. By applying the Cauchy's inequality we obtain
\begin{multline*}
J_1 \leqslant \int_1^{+\infty}\abs{\sum_{n=0}^{+\infty}d_\beta(n)e^{-nxe^{i\left (\frac{\pi}{2}-\delta \right )}}}^2dx + \int_1^{+\infty}\abs{\Psi_\beta \left (x e^{i \left (\frac{\pi}{2}-\delta \right )} \right )}^2dx
\\ + O \left ( \left (\int_1^{+\infty}\abs{\sum_{n=0}^{+\infty}d_\beta(n)e^{-nxe^{i\left (\frac{\pi}{2}-\delta \right )}}}^2dx \right )^{1/2} \left (  \int_1^{+\infty}\abs{\Psi_\beta \left (x e^{i \left (\frac{\pi}{2}-\delta \right )} \right )}^2dx \right )^{1/2} \right ).
\end{multline*}
One has
\[\abs{\Psi_\beta \left (x e^{i \left (\frac{\pi}{2}-\delta \right )} \right )} \leqslant \frac{1}{x}\sum_{m=1}^{\beta}\frac{\abs{\lambda_m}}{(m-1)!}\left(\log x + \frac{\pi}{2}\right)^{m-1}.\]
From \cite[Lem. 3]{Tit28} there exists a positive constant $A_1 > 0$ such that
\[\abs{\Psi_\beta \left (x e^{i \left (\frac{\pi}{2}-\delta \right )} \right )} \leqslant \beta\exp(A_1 \beta \log \beta)\frac{\left(\log x + \frac{\pi}{2} \right)^{\beta-1}}{x}.\]
Thus
\[\int_1^{+\infty}\abs{\Psi_\beta \left (x e^{i \left (\frac{\pi}{2}-\delta \right )} \right )}^2dx = O_\beta\left(\int_1^{+\infty}\frac{\log\left(x + \frac{\pi}{2} \right)^{2(\beta-1)}}{x^2}dx \right) = O_{\beta,\varepsilon}(1).\]
Now we compute
\begin{multline*}
\int_1^{+\infty}\abs{\sum_{n=0}^{+\infty}d_\beta(n)e^{-nxe^{i\left (\frac{\pi}{2}-\delta \right )}}}^2dx = \int_1^{+\infty}\sum_{n=1}^{+\infty}\sum_{m=1}^{+\infty}d_\beta(n)d_\beta(m)\exp\left (inxe^{i\delta} - imxe^{-i\delta} \right )dx \\ = \sum_{n=1}^{+\infty}\sum_{m=1}^{+\infty}d_\beta(n)d_\beta(m)\int_1^{+\infty}\exp \left (inxe^{i\delta} - imxe^{-i\delta} \right ) = \sum_{n=1}^{+\infty}d_\beta^2(n)\int_1^{+\infty}e^{-2nx\sin \delta}dx \\ + \sum_{n \neq m}d_\beta(n)d_\beta(m)\int_1^{+\infty}\exp \left ( -(n+m)x\sin \delta + i(n-m)x \cos \delta \right ) dx \\ = \Sigma_1 + \Sigma_2.
\end{multline*}
One has 
\[\Sigma_1 = \frac{1}{2\sin \delta}\sum_{n=1}^{+\infty}\frac{d^2_\beta(n)}{n}e^{-2n\sin \delta} \sim \frac{C_\beta}{\delta}\log^{\beta^2}\left(\frac{1}{\delta}\right), \quad \delta \to 0^+.\]
Indeed, one introduces the function $F_\beta(s) = \sum_{n=1}^{+\infty}\frac{d_\beta^2(n)}{n^s}$ for $\sigma > 1$. Then one can write $F_\beta(s) = \zeta^{\beta^2}(s)g(s)$, where $g(s)$ is an analytic function in $\sigma > \frac{1}{2}$ (cf. \cite[p. 174]{Tit86}). It means
\[\sum_{n=1}^{+\infty}\frac{d^2_\beta(n)}{n}e^{-2n\sin \delta} = \frac{1}{2\pi i}\int_{2-i\infty}^{2+i\infty}\Gamma(s)F_\beta(s+1)\frac{ds}{(2\sin \delta)^s},\] 
where the function $\Gamma(s)F_\beta(s+1)$ has a pole of order $\beta^2+1$ at $s=0$. Residue calculation and moving countour integral give us then the asymptotic formula of $\Sigma_1$. To estimate $\Sigma_2$ we will follow the idea in \cite[proof of Lem. 5]{Tit28}. One has for $n>m$
\begin{multline*}
\abs{\int_1^{+\infty}\exp\left(-(n+m)x\sin \delta + i(n-m)x\cos \delta \right) dx} \\ = \abs{\int_0^{+\infty}\exp \left(-(n+m)(1+iy)\sin \delta + i(n-m)(1+iy)\cos \delta \right ) dy} \\  < e^{-(n+m)\sin \delta}\int_0^{+\infty}e^{-(n-m)y \cos \delta}dy  < \frac{e^{-n\sin \delta}}{(n-m)\cos \delta}.
\end{multline*}
By using the symmetry between the indices $n$ and $m$ in $\Sigma_2$ we have consequently
\begin{multline*}
\Sigma_2 \ll \frac{1}{\cos \delta}\sum_{n=2}^{+\infty}\sum_{m=1}^{n-1}d_\beta(n)d_\beta(m)\frac{1}{(n-m)}e^{-n\sin \delta} = \frac{1}{\cos \delta}\sum_{k=1}^{+\infty}\sum_{n=k+1}^{+\infty}d_\beta(n)d_\beta(n-k)\frac{e^{-n\sin\delta}}{k} \\ = \frac{1}{\cos \delta}\sum_{k=1}^{+\infty}\frac{e^{-\frac{1}{2}k\sin \delta}}{k}\sum_{n=k+1}^{+\infty}d_\beta(n)e^{-\frac{1}{2}n\sin \delta}d_\beta(n-k)e^{-\frac{1}{2}(n-k)\sin \delta}.
\end{multline*}
The Cauchy's inequality yields
\begin{multline*}
\sum_{n=k+1}^{+\infty}d_\beta(n)e^{-\frac{1}{2}n\sin \delta}d_\beta(n-k)e^{-\frac{1}{2}(n-k)\sin \delta} \leqslant \\ \left (\sum_{n=k+1}^{+\infty}d_\beta^2(n)e^{-n\sin \delta}\sum_{n=k+1}^{+\infty}d_\beta^2(n-k)e^{-(n-k)\sin \delta} \right )^{\frac{1}{2}} < \sum_{n=1}^{+\infty}d_\beta^2(n)e^{-n\sin \delta}.
\end{multline*}
One has
\[\sum_{k=1}^{+\infty}\frac{e^{-\frac{1}{2}k\sin^2 \delta}}{k} = \log \frac{1}{1-e^{-\frac{1}{2}\sin \delta}} \sim \log \frac{2}{\delta} \ll_\varepsilon \frac{1}{\delta^\varepsilon}, \quad \forall \varepsilon > 0, \delta \to 0.\]
Now one has (cf. \cite[Chap. VII, p. 174]{Tit86})
\[\sum_{n=1}^{+\infty}d^2_\beta(n)e^{-n\sin \delta} \sim \frac{C_\beta}{\delta}\log^{\beta^2-1}\frac{1}{\delta}, \quad \delta \to 0^+, \forall \varepsilon > 0.\]
Indeed, by analytic continuation of the Cahen-Mellin's formula for $\mathfrak{R}(z) > 0$ and $\left|\mathrm{Arg}(z) < \frac{\pi}{2} \right|$ one has
\[\sum_{n=1}d_\beta^2(n)e^{-nz} = \frac{1}{2\pi i}\int_{2-i\infty}^{2+i\infty}\Gamma(s)F_\beta(s)\frac{ds}{z^s}.\]
The function $F_\beta(s)$ has a pole at $s=1$. By writing
\[\frac{1}{z^s} = \frac{1}{z}\exp\left(-(s-1)\log z \right) = \frac{1}{z}\sum_{n=0}^{+\infty}\frac{(-1)^n}{n!}(s-1)^n\log^n z,\]
one gets the asymptotic formula above by residue calculation and Stirling's formula. This means
\[\Sigma_2 \ll_\beta \frac{1}{\delta}\log^{\beta^2-1}\left(\frac{1}{\delta}\right), \quad \delta \to 0^+.  \]
So we obtain the following result, which is well-known to the experts:
\[J_1 \sim \frac{C_\beta}{\delta}\log^{\beta^2}\left(\frac{1}{\delta}\right), \quad \delta \to 0^+, \forall \varepsilon > 0, \forall \beta \geqslant 2.\]
Now it remains to estimate the integral
\begin{multline*}J_2 = \int_0^1\abs{\Phi_\beta\left (xe^{i\left (\frac{\pi}{2} - \delta \right )} \right )}^2dx =
\int_0^{\delta^{\frac{\beta-1}{2}}}\abs{\Phi_\beta\left (xe^{i\left (\frac{\pi}{2} - \delta \right )} \right )}^2dx + \int_{\delta^{\frac{\beta-1}{2}}}^1\abs{\Phi_\beta\left (xe^{i\left (\frac{\pi}{2} - \delta \right )} \right )}^2dx \\ = J_3 + J_4
\end{multline*}
By changing variable $x = \frac{1}{y}$ we have
\[J_3 = \int_{\frac{1}{\delta^{\frac{\beta-1}{2}}}}^{+\infty}\abs{\Phi_\beta \left (\frac{e^{i \left (\frac{\pi}{2} - \delta \right )}}{y} \right )}^2 \frac{dy}{y^2}.\]
The estimation of $J_3$ is well-known. However we need to spell out a few words in details. We consider the function
\[\Phi_\beta \left(\frac{1}{\bar{z}}\right) = \frac{1}{2\pi i}\int_{\alpha-i\infty}^{\alpha+i\infty}\Gamma(s)\zeta^{\beta}(s)\bar{z}^s ds, \quad 0 < \alpha < 1.\]
By the Cauchy's integration theorem we have
\[\mathsf{Res}_{s=0}\left (\Gamma(s)\zeta^{\beta}(s)\bar{z}^s \right ) = \frac{1}{2\pi i}\oint_{D}\Gamma(s)\zeta^\beta(s)\bar{z}^sds,\]
where $D$ is the rectangle $[\alpha-iT,\alpha+iT,-\alpha+iT,-\alpha-iT]$, since $\Gamma(s)$ has a simple pole at $s=0$. As before, by using the Stirling's formula, we see that the integrals along the horizontal lines of $D$ vanish, when $T \to +\infty$. Residue calculation gives us
\[\mathsf{Res}_{s=0}\left (\Gamma(s)\zeta^{\beta}(s)\bar{z}^s \right ) = O_{\beta}(1).\]
One has then
\[\Phi_\beta\left (\frac{1}{\bar{z}} \right ) = \frac{1}{2\pi i}\int_{-\alpha-i\infty}^{-\alpha+i\infty}\Gamma(s)\zeta^\beta(s)\bar{z}^sds + O_\beta(1).\] 
After changing the variable $s \to 1-s$ one has
\[\Phi_\beta \left (\frac{1}{\bar{z}} \right ) = \frac{\bar{z}}{2\pi i}\int_{1+\alpha-i\infty}^{1+\alpha + i\infty}\Gamma(1-s)\zeta^\beta(1-s) \frac{ds}{\bar{z}^s} + O_\beta(1).\]
By using the functional equation of the Riemann zeta function (cf. \cite[p. 9]{Iv85})
\[\zeta(s) = \chi(s)\zeta(1-s), \quad \chi(s) = (2\pi)^s/(2\Gamma(s)\cos(\pi s/2))\]
one obtains
\[\Phi_\beta \left (\frac{1}{\bar{z}} \right ) = \frac{\bar{z}}{2\pi i}\int_{1+\alpha - i\infty}^{1 + \alpha + i\infty} \frac{\Gamma(1-s)}{\chi^\beta (s)}\zeta^\beta(s) \frac{ds}{\bar{z}^s} + O_\beta(1).\]
The Euler's functional equation
\[\Gamma(s) \Gamma(1-s) = \frac{\pi}{\sin \pi s}\]
gives us
\[\Phi_\beta \left ( \frac{1}{\bar{z}} \right ) = \frac{\bar{z}}{2\pi i}\int_{1+\alpha-i\infty}^{1+\alpha+i\infty}2^{\beta(1-s)}\pi^{1-\beta s}\frac{\cos^\beta \left (\frac{\pi s}{2} \right )}{\sin \pi s}\Gamma^{\beta-1}(s)\zeta^\beta(s) \frac{ds}{\bar{z}^s} + O_\beta(1).\]
On the vertical line $\sigma = 1+\alpha$ we have $\zeta^\beta(s) = \sum_{n=1}^{+\infty}\frac{d_\beta(n)}{n^s}$, so we can rewrite
\[\Phi_\beta \left ( \frac{1}{\bar{z}} \right ) = \frac{\bar{z}}{2\pi i}\sum_{n=1}^{+\infty}d_\beta(n)\int_{1+\alpha-i\infty}^{1+\alpha+i\infty}2^{\beta(1-s)}\pi^{1-\beta s}\frac{\cos^\beta \left (\frac{\pi s}{2} \right )}{\sin \pi s}\Gamma^{\beta-1}(s)\frac{ds}{(n\bar{z})^s} + O_\beta(1).\]
We consider the integral
\[I = \int_{1+\alpha-i\infty}^{1+\alpha+i\infty}2^{\beta(1-s)}\pi^{1-\beta s}\frac{\cos^\beta \left (\frac{\pi s}{2} \right )}{\sin \pi s}\Gamma^{\beta-1}(s)\frac{ds}{(n\bar{z})^s}.\]
By the Stirling's formula we have $\exists t_1 > 0$, such that $\forall \abs{t} \geqslant t_1$
\[\abs{\Gamma\left(1+\alpha+it\right)}^{\beta-1} \sim (2\pi)^{\frac{\beta-1}{2}}\abs{t}^{\left(\frac{1}{2}+\alpha\right)(\beta-1)}e^{-\frac{\pi}{2}\abs{t}(\beta-1)}(1+O(\abs{t}^{-1}))^{\beta-1}.\]
Again the Stirling's formula shows that $\exists t_2 > 0$, such that $\forall \abs{t} \geqslant t_2$ and for $s = 1+\alpha+it$
\[\abs{\Gamma\left((\beta-1)s-\frac{1}{2}\beta+1 \right)} \sim \sqrt{2\pi}((\beta-1)\abs{t})^{\left(\frac{1}{2}+\alpha\right)(\beta-1)}e^{-\frac{\pi}{2}\abs{t}(\beta-1)}(1+O(\abs{t}^{-1})).\]
So we must have
\[\Gamma^{\beta-1} \left( s \right) \ll_{\beta} \Gamma\left((\beta-1)s-\frac{1}{2}\beta+1 \right).\]
So $\exists t_3 = \max(t_1,t_2)$, such that 
\begin{multline*}
I = \left(\int_{1+\alpha-i\infty}^{1+\alpha-it_3} + \int_{1+\alpha-it_3}^{1+\alpha+it_3} + \int_{1+\alpha+it_3}^{1+\alpha+i\infty} \right)2^{\beta(1-s)}\pi^{1-\beta s}\frac{\cos^\beta \left (\frac{\pi s}{2} \right )}{\sin \pi s}\Gamma^{\beta-1}(s)\frac{ds}{(n\bar{z})^s} \\ \ll_{\beta,\varepsilon} O_{\beta,\varepsilon}\left ( \frac{1}{(ny)^{1+\alpha}} \right ) + \int_{1+\alpha-i\infty}^{1+\alpha+i\infty}2^{\beta(1-s)}\pi^{1-\beta s}\frac{\cos^\beta \left (\frac{\pi s}{2} \right )}{\sin \pi s}\Gamma\left((\beta-1)s -\frac{1}{2}\beta +1 \right)\frac{ds}{(n\bar{z})^s} \\ = O_{\beta,\varepsilon}\left( \frac{1}{(ny)^{1+\alpha}} \right) + \int_{1+\alpha-i\infty}^{1+\alpha+i\infty}2^{\beta(1-s)-1}\pi^{1-\beta s}\frac{\cos^{\beta-1}\left(\frac{\pi s}{2} \right)}{\sin \left ( \frac{\pi s}{2} \right )}\Gamma\left((\beta-1)s -\frac{1}{2}\beta +1 \right)\frac{ds}{(n\bar{z})^s}.
\end{multline*}
In the upper-half plane $t \geqslant 0$ we have
\[\frac{\cos^{\beta-1}\left(\frac{\pi s}{2} \right)}{\sin \left (\frac{\pi s}{2} \right )} = \frac{\frac{1}{2^{\beta-1}} \left (e^{\frac{i \pi s}{2}} + e^{\frac{-i\pi s}{2}} \right )^{\beta-1}}{\frac{1}{2i} \left ( e^{\frac{i \pi s}{2}} - e^{\frac{-i\pi s}{2}} \right )} \sim -2^{2-\beta}ie^{-i(\beta-2)\frac{\pi s}{2}}, \quad t \to +\infty\]
and for $t < 0$ we have
\[\frac{\cos^{\beta-1}\left(\frac{\pi s}{2} \right)}{\sin \left (\frac{\pi s}{2} \right )} = \frac{\frac{1}{2^{\beta-1}} \left (e^{\frac{i \pi s}{2}} + e^{\frac{-i\pi s}{2}} \right )^{\beta-1}}{\frac{1}{2i} \left ( e^{\frac{i \pi s}{2}} - e^{\frac{-i\pi s}{2}} \right )} \sim 2^{2-\beta}ie^{i(\beta-2)\frac{\pi s}{2}}, \quad t \to -\infty.\]
Now for $z = ye^{i \left (\frac{\pi}{2}-\delta \right )}$ with $0 < \delta < \frac{\pi}{2}$
\begin{multline*}
\int_{1+\alpha-i\infty}^{1+\alpha}e^{-i(\beta-2)\frac{\pi s}{2}}\Gamma\left((\beta-1)s - \frac{1}{2}\beta + 1 \right)\frac{ds}{(n\bar{z})^s} = \\ i\int_{-\infty}^0e^{-i\frac{\pi}{2}(1+\alpha)(\beta-2)}e^{\frac{\pi}{2}t(\beta-2)}\Gamma\left((\beta-1)\left(1+\alpha+it \right)-\frac{1}{2}\beta + 1 \right)\frac{dt}{(n\bar{z})^{1+\alpha+it}} \\ \ll_{\beta,\varepsilon} \frac{1}{(ny)^{1+\alpha}}\int_{-\infty}^0\abs{t}^{\left(\frac{1}{2}+\alpha\right)(\beta-1)}e^{\frac{\pi}{2}t(\beta-1)}e^{\frac{\pi}{2}t(\beta-1)}e^{t\left(\delta - \frac{\pi}{2} \right )}dt = O_{\beta,\varepsilon}\left(\frac{1}{(ny)^{1+\alpha}}\right).
\end{multline*}
Similarly
\[\int_{1+\alpha}^{1+\alpha+i\infty}e^{i(\beta-2)\frac{\pi s}{2}}\Gamma\left((\beta-1)s - \frac{1}{2}\beta + 1 \right)\frac{ds}{(n\bar{z})^s} \ll_{\beta,\varepsilon} \frac{1}{(ny)^{1+\alpha}}.\]
So we obtain
\begin{multline*}
I \ll_{\beta,\varepsilon} O_{\beta,\varepsilon}\left(\frac{1}{(ny)^{1+\alpha}} \right) \\ + \int_{1+\alpha-i\infty}^{1+\alpha+i\infty}(2\pi)^{1-\beta s}\left(e^{i(\beta-2)\frac{\pi s}{2}} + e^{-i(\beta-2)\frac{\pi s}{2}} \right)\Gamma\left((\beta-1)s -\frac{1}{2}\beta +1 \right)\frac{ds}{(n\bar{z})^s}.
\end{multline*}
We set $(\beta-1)s-\frac{1}{2}\beta+1 = w$, the integral becomes
\begin{multline*}
\int_{1+\alpha-i\infty}^{1+\alpha+i\infty}(2\pi)^{1-\beta s}\left(e^{i(\beta-2)\frac{\pi s}{2}} + e^{-i(\beta-2)\frac{\pi s}{2}} \right)\Gamma\left((\beta-1)s -\frac{1}{2}\beta +1 \right)\frac{ds}{(n\bar{z})^s} \\ = \frac{1}{\beta-1}\int_{(\beta-1)(1+\alpha)-\beta/2+1-i\infty}^{(\beta-1)(1+\alpha)-\beta/2+1+i\infty}(2\pi)^{1-\frac{\beta}{\beta-1}\left(w + \frac{1}{2}\beta -1\right)}\Gamma(w) \times \\ \times \left (e^{\frac{i \pi(\beta-2)(w+\beta/2-1)}{2(\beta-1)}} + e^{-\frac{i \pi(\beta-2)(w+\beta/2-1)}{2(\beta-1)}} \right)\frac{dw}{(n\bar{z})^{\frac{w+\beta/2-1}{\beta-1}}}.
\end{multline*}
We can write
\begin{multline*}
\int_{(\beta-1)(1+\alpha)-\beta/2+1-i\infty}^{(\beta-1)(1+\alpha)-\beta/2+1+i\infty}(2\pi)^{1-\frac{\beta}{\beta-1}\left(w + \frac{1}{2}\beta -1\right)}\Gamma(w)e^{\frac{i \pi(\beta-2)(w+\beta/2-1)}{2(\beta-1)}}\frac{dw}{(n\bar{z})^{\frac{w+\beta/2-1}{\beta-1}}} \\ = \frac{2\pi}{\left( (2\pi)^\beta n \bar{z}e^{-i(\beta-2)\frac{\pi}{2}} \right)^{\frac{\beta/2-1}{\beta-1}}} \int_{(\beta-1)(1+\alpha)-\beta/2+1-i\infty}^{(\beta-1)(1+\alpha)-\beta/2+1+i\infty}\frac{\Gamma(w)dw}{\left(\left((2\pi)^\beta n \bar{z}e^{-i(\beta-2)\frac{\pi}{2}} \right)^{\frac{1}{\beta-1}}\right)^w}.
\end{multline*}
By Cahen-Mellin's formula we have
\[\int_{(\beta-1)(1+\alpha)-\beta/2+1-i\infty}^{(\beta-1)(1+\alpha)-\beta/2+1+i\infty}\frac{\Gamma(w)dw}{\left(\left((2\pi)^\beta n \bar{z}e^{-i(\beta-2)\frac{\pi}{2}} \right)^{\frac{1}{\beta-1}}\right)^w} = \exp \left(- \left ((2\pi)^\beta n \bar{z}e^{-i(\beta-2)\frac{\pi}{2}} \right )^{\frac{1}{\beta-1}} \right).\]
Similary
\begin{multline*}
\int_{(\beta-1)(1+\alpha)-\beta/2+1-i\infty}^{(\beta-1)(1+\alpha)-\beta/2+1+i\infty}(2\pi)^{1-\frac{\beta}{\beta-1}\left(w + \frac{1}{2}\beta -1\right)}\Gamma(w)e^{-\frac{i \pi(\beta-2)(w+\beta/2-1)}{2(\beta-1)}}\frac{dw}{(n\bar{z})^{\frac{w+\beta/2-1}{\beta-1}}} \\ = \frac{2\pi}{\left( (2\pi)^\beta n \bar{z}e^{i(\beta-2)\frac{\pi}{2}} \right)^{\frac{\beta/2-1}{\beta-1}}}\exp \left(- \left ((2\pi)^\beta n \bar{z}e^{i(\beta-2)\frac{\pi}{2}} \right )^{\frac{1}{\beta-1}} \right).
\end{multline*}
Thus we have for $z = ye^{i\left(\frac{\pi}{2}-\delta\right)}$
\begin{multline*}
I = O_{\beta,\varepsilon}\left(\frac{1}{(ny)^{1+\alpha}} \right)+\frac{2\pi\exp \left(-\left((2\pi)^\beta  n ye^{i\left (\delta + \frac{\pi}{2} - \frac{\pi \beta}{2} \right )}  \right)^{\frac{1}{\beta-1}} \right)}{(\beta-1)\left((2\pi)^\beta n y e^{i \left(\delta + \frac{\pi}{2} - \frac{\pi \beta}{2} \right)} \right)^{\frac{\beta/2-1}{\beta-1}}} \\ + \frac{2\pi \exp \left (-\left ( (2\pi)^\beta n y e^{i\left ( \delta - \frac{3\pi}{2} + \frac{\pi\beta}{2} \right )} \right)^{\frac{1}{\beta-1}} \right)}{(\beta-1)\left((2\pi)^\beta n y e^{i\left (\delta - \frac{3\pi}{2} + \frac{\pi \beta}{2} \right )} \right)^{\frac{\beta/2-1}{\beta-1}}}.
\end{multline*}
On the other hand
\[e^{i\frac{\delta + \frac{\pi}{2}-\frac{\pi\beta}{2}}{\beta-1}} = \sin\left (\frac{\delta}{\beta-1}\right ) -i \cos \left (\frac{\delta}{\beta-1} \right ) \]
and $\sin \left(\frac{\delta}{\beta-1} \right) > 0$ for $0 < \delta < \frac{\pi}{2}$. Similarly
\[e^{i\frac{\delta - \frac{3\pi}{2} + \frac{\pi \beta}{2}}{\beta-1}} = \sin \left (\frac{\pi-\delta}{\beta-1} \right ) + i\cos \left (\frac{\pi-\delta}{\beta-1} \right) \]
and $\sin \left(\frac{\pi - \delta}{\beta-1} \right) > 0$ for $0 < \delta < \frac{\pi}{2}$ and $\beta \geqslant 2$. Now we can estimate
\begin{multline*}
\Phi_\beta \left(\frac{1}{y}e^{i\left(\delta - \frac{\pi}{2} \right)} \right) \\ \ll_{\beta,\varepsilon} \frac{1}{y^{\alpha}}\sum_{n=1}^{+\infty}\frac{d_\beta(n)}{n^{1+\alpha}} + \frac{1}{y^{\frac{\beta/2-1}{\beta-1}-1}}\sum_{n=1}^{+\infty}d_\beta(n)\frac{\exp\left(i(2\pi)^{\frac{\beta}{\beta-1}}n^{\frac{1}{\beta-1}}y^{\frac{1}{\beta-1}}\exp\left (i\frac{\delta}{\beta-1}\right) \right)}{n^{\frac{\beta/2-1}{\beta-1}}} \\ + \frac{1}{y^{\frac{\beta/2-1}{\beta-1}-1}}\sum_{n=1}^{+\infty}d_\beta(n)\frac{\exp\left(-(2\pi)^{\frac{\beta}{\beta-1}}n^{\frac{1}{\beta-1}}y^{\frac{1}{\beta-1}}\sin \left (\frac{\pi-\delta}{\beta-1}\right) \right)}{n^{\frac{\beta/2-1}{\beta-1}}} = T_1 + T_2 + T_3.
\end{multline*}
The main contributor of $J_3$ comes from the middle term $T_2$ of the estimation above, because the estimation of the mixed terms $\int_{\frac{1}{\delta^{\frac{\beta-1}{2}}}}^{+\infty}T_1T_2\frac{dy}{y^2}$ and $\int_{\frac{1}{\delta^{\frac{\beta-1}{2}}}}^{+\infty}T_2T_3\frac{dy}{y^2}$ follows immediately then by using Cauchy's inequality, as well the terms $\int_{\frac{1}{\delta^{\frac{\beta-1}{2}}}}^{+\infty}T_1 \frac{dy}{y^2}$ and $\int_{\frac{1}{\delta^{\frac{\beta-1}{2}}}}^{+\infty}T_3^2\frac{dy}{y^2}$ are bounded when $\delta \to 0^+$ for $\beta \geqslant 3$. For $\beta = 2$ all terms $T_2$ and $T_3$ are asymptotic equivalent. Now we are allowed to integrate and sum up like in \cite[Chap. VII, p. 172-173]{Tit86} to obtain 
\[J_3 = \int_{\frac{1}{\delta^{\frac{\beta-1}{2}}}}^{+\infty}\abs{\Phi_\beta\left(\frac{1}{y}ie^{-i\delta}\right)}^2\frac{dy}{y^2} \ll \int_{\frac{1}{\delta^{\frac{\beta-1}{2}}}}^{+\infty}\left|T_2\right|^2 \frac{dy}{y^2} = O_{\beta,\varepsilon}\left(\frac{1}{\delta^{\frac{\beta-1}{2}+\varepsilon}}\right).\]
It remains to estimate the integral
\[J_4 = \int_{\delta^{\frac{\beta-1}{2}}}^1 \abs{\Phi_\beta (xie^{-i\delta})}^2dx.\]
Recall, by residue calculation one has for $\mathfrak{R}(z) > 0$ and $\left|\mathrm{Arg}(z)\right| < \frac{\pi}{2} $ (cf. \cite[Chap. VII \S 7.13, p. 160]{Tit86})
\[\Phi_\beta(z) = \sum_{n=1}^{+\infty}d_\beta(n)e^{-nz}+\frac{1}{z}\sum_{n=0}^{\beta-1}a^{(\beta)}_n\log^nz.\]
Cauchy's inequality gives us
\begin{multline*}
J_4 = O\left(\int_{\delta^{\frac{\beta-1}{2}}}^1\left|\sum_{n=1}^{+\infty}d_\beta(n)e^{-nxie^{-i\delta}} \right|^2dx \right) + O\left(\int_{\delta^{\frac{\beta-1}{2}}}^1\left|\sum_{n=0}^{\beta-1}a_n^{(\beta)}\left(\log x + i \left(\frac{\pi}{2}-\delta \right) \right)^n \right|^2 \frac{dx}{x^2} \right) \\ = J_5+J_6.
\end{multline*}
The estimation of $J_6$ is straightforward. Again by Cauchy's inequality
\begin{multline*}
J_6 \leqslant \int_{\delta^{\frac{\beta-1}{2}}}^1\left(\left|a^{(\beta)}_n\right| \left|\log x + i\left(\frac{\pi}{2}-\delta \right)\right|^n \right)^2 \frac{dx}{x^2} \leqslant \beta \int_{\delta^{\frac{\beta-1}{2}}}^1\left|a_n^{(\beta)}\right|^2 \left|\log x + i\left(\frac{\pi}{2}-\delta\right) \right|^{2n}\frac{dx}{x^2} \\ = \beta\sum_{n=0}^{\beta-1}\left|a_n^{(\beta)}\right|\int_1^{\frac{1}{\delta^{\frac{\beta-1}{2}}}}\left(\log^2 y + \left(\frac{\pi}{2}-\delta \right)^2 \right)^n dy = O\left(\frac{1}{\delta^{\frac{\beta-1}{2}+\varepsilon}}\right), \quad \delta \to 0^+, \forall \varepsilon > 0.
\end{multline*}
Let $\alpha \in (0,1)$ be a real number. We have for $z = xie^{-i\delta}$
\begin{multline*}
J_5 = \int_{\delta^{\frac{\beta-1}{2}}}^1 \left|\sum_{n=1}^{+\infty}d_\beta(n)e^{-nz} \right|^2dx \\ =  \int_{\delta^{\frac{\beta-1}{2}}}^1\left|\mathsf{Res}_{s=1}\frac{\Gamma(s)\zeta^\beta(s)}{z^s} + \mathsf{Res}_{s=0}\frac{\Gamma(s)\zeta^\beta(s)}{z^s} -\frac{1}{2\pi i}\int_{-\alpha-i\infty}^{-\alpha+i\infty}\Gamma(s)\zeta^\beta(s)\frac{ds}{z^s} \right|^2dx
\\ \ll \int_{\delta^{\frac{\beta-1}{2}}}^1  \left|\mathsf{Res}_{s=1}\frac{\Gamma(s)\zeta^\beta(s)}{z^s}\right|^2dx + \int_{\delta^{\frac{\beta-1}{2}}}^1  \left|\mathsf{Res}_{s=0}\frac{\Gamma(s)\zeta^\beta(s)}{z^s}\right|^2dx \\ + \int_{\delta^{\frac{\beta-1}{2}}}^1 \left|\frac{1}{2\pi i}\int_{-\alpha-i\infty}^{-\alpha+i\infty}\Gamma(s)\zeta^\beta(s)\frac{ds}{z^s} \right|^2 dx.
\end{multline*}
By residue calculation as before we have
\[\mathsf{Res}_{s=1}\frac{\Gamma(s)\zeta^\beta(s)}{\left(xie^{-i\delta}\right)^s} = \frac{1}{xie^{-i\delta}}P_{\beta-1}\left(\log \frac{1}{xie^{-i\delta}} \right) \sim \frac{C_\beta}{x}\log^{\beta-1}\left(\frac{1}{x} \right),\]
where we denote by $P_{\beta-1}(z)$ a polynomial of degree $\beta-1$ in $z$ and
\[\mathsf{Res}_{s=0}\frac{\Gamma(s)\zeta^\beta(s)}{\left(xie^{-i\delta}\right)^s} = O(1).\]
So
\[\int_{\delta^{\frac{\beta-1}{2}}}^1\left|\mathsf{Res}_{s=1}\frac{\Gamma(s)\zeta^\beta(s)}{\left(xie^{-i\delta} \right)^s}\right|^2 dx \sim C_\beta\int_{\delta^{\frac{\beta-1}{2}}}^1 \log^{\beta-1}\left(\frac{1}{x} \right) \frac{dx}{x^2} = O_{\beta,\varepsilon}\left(\frac{1}{\delta^{\frac{\beta-1}{2}+\varepsilon}}\right)\]
and
\[\int_{\delta^{\frac{\beta-1}{2}}}^1\left|\mathsf{Res}_{s=0}\frac{\Gamma(s)\zeta^\beta(s)}{\left(xie^{-i\delta} \right)^s}\right|^2 dx = O(1).\]
The functional equation $\zeta(s) = \chi(s)\zeta(1-s)$ gives us
\[\int_{-\alpha-i\infty}^{-\alpha+i\infty}\Gamma(s)\zeta^\beta(s)\frac{ds}{z^s} = \sum_{n=1}^{+\infty}d_\beta(n)\int_{-\alpha-i\infty}^{-\alpha+i\infty}\Gamma(s)\chi^\beta(s)\frac{ds}{n^{1-s}z^s}.\]
From Stirling's formula 
\[\Gamma(s) \sim \sqrt{2\pi}s^{s-\frac{1}{2}}e^{-s}, \quad \left|s\right| \to +\infty,\]
one has (see e.g \cite[\S 1.2, p. 9]{Iv85})
\[\chi(s) = \left(\frac{2\pi}{t}\right)^{\sigma+it-\frac{1}{2}}e^{i\left(t+\frac{\pi}{4}\right)}\left(1 + O\left(\frac{1}{t}\right)\right), t \geqslant t_0 > 0.\]
We see for $z = xie^{-i\delta}$
\[\int_{-\alpha-i\infty}^{-\alpha}\Gamma(s)\chi^\beta(s)\frac{ds}{n^{1-s}z^s} = O\left(\frac{1}{n^{1+\alpha}} \right).\]
For $t \geqslant \frac{1}{\delta^{1+\varepsilon}}$ one has
\begin{multline*}
\int_{-\alpha+i\frac{1}{\delta^{1+\varepsilon}}}^{-\alpha+i\infty}\Gamma(s)\chi^\beta(s)\frac{ds}{n^{1-s}z^s} \ll \frac{1}{n^{1+\alpha}}\int_{\frac{1}{\delta^{1+\varepsilon}}}^{+\infty}t^{(\beta-1)(\alpha+1/2)}e^{-\delta t}dt \\ \leqslant \frac{e^{-\frac{1}{2 \delta^\varepsilon}}}{n^{1+\alpha}}\int_{\frac{1}{\delta^{1+\varepsilon}}}^{+\infty}t^{(\beta-1)(\alpha+1/2)}e^{-\frac{1}{2}\delta t}dt = o\left(\frac{1}{n^{1+\alpha}}\right).
\end{multline*}
Now we can write for $t_0 > 0$
\begin{multline*}
\int_{-\alpha}^{-\alpha+i\frac{1}{\delta^{1+\varepsilon}}}\Gamma(s)\chi^\beta(s)\frac{ds}{n^{1-s}z^s} = O\left(\frac{1}{n^{1+\alpha}}\right) \\ + \frac{ix^\alpha}{n^{1+\alpha}}\int_{t_0}^{\frac{1}{\delta^{1+\varepsilon}}}\Gamma(-\alpha+it)\chi^\beta(-\alpha+it)\exp\left\{-it \log n - \left(-\alpha+it\right)\left(x + i\left(\frac{\pi}{2}-\delta \right) \right) \right\}dt \\ \ll_{\alpha}  \frac{1}{n^{1+\alpha}} + \frac{x^\alpha}{n^{1+\alpha}}\int_{t_0}^{\frac{1}{\delta^{1+\varepsilon}}}t^{\left(\beta-1\right)\left(\alpha+\frac{1}{2} \right)}e^{-\delta t}\exp \left\{iF(t)\right\}dt,
\end{multline*}
where
\[F(t) = t\left(\log t - 1 - \beta\log t + \beta\log 2 \pi + \beta - \log n - \log x \right)\]
and the asymptotic formula of $\Gamma(s)$ with $\mathfrak{R}(s) < 0$ follows easily from the reflection formula $\Gamma(s)\Gamma(1-s) = \frac{\pi}{\sin \pi s}$.
Indeed, one has in any fixed strip $\alpha_1 \leqslant -\alpha \leqslant \alpha_2$
\[\Gamma(-\alpha+it) = \sqrt{2\pi}t^{-\alpha+it-\frac{1}{2}}e^{-\frac{\pi}{2}t-it-\frac{1}{2}i \pi\left(\alpha+\frac{1}{2}\right)}\left\{1 + O\left(\frac{1}{t} \right) \right\} \]
and
\[\chi\left(-\alpha+it\right) = \left(\frac{2\pi}{t}\right)^{-\alpha-\frac{1}{2}+it}e^{i\left(t+\frac{\pi}{4}\right)}\left\{1 + O\left(\frac{1}{t} \right) \right\}.\]
As 
\[F^{\prime \prime}(t) = \frac{1-\beta}{t}\]
we can apply \cite[Lem. 2.2]{Iv85} for $G(t) = t^{\left(\beta-1\right)\left(\alpha+\frac{1}{2} \right)}e^{-\delta t}$ by splitting up
\begin{multline*}
\int_{t_0}^{\frac{1}{\delta^{1+\varepsilon}}}G(t)\exp\left\{iF(t)\right\}dt = \\ \int_{t_0}^{\frac{(\beta-1)\left(\alpha+\frac{1}{2}\right)}{\delta}}G(t)\exp\left\{iF(t)\right\}dt + \int_{\frac{(\beta-1)\left(\alpha+\frac{1}{2}\right)}{\delta}}^{\frac{1}{\delta^{1+\varepsilon}}}G(t)\exp\left\{iF(t)\right\}dt \\ \ll_{\beta,\alpha}\frac{1}{\delta^{\left(\beta-1 \right)\left(\alpha + \frac{1}{2} \right)-\frac{1}{2}}}+\frac{1}{\delta^{\left(\beta-1 \right)\left(\alpha + \frac{1}{2} \right) - \frac{1+\varepsilon}{2}}} = O_{\beta,\alpha}\left(\frac{1}{\delta^{\left(\beta-1 \right)\left(\alpha + \frac{1}{2} \right)-\frac{1}{2}}}\right),
\end{multline*}
because $G(t)$ is monotonic positive on each interval above. Putting all this together, we obtain
\[\int_{\delta^{\frac{\beta-1}{2}}}^1 \left|\frac{1}{2\pi i}\int_{-\alpha-i\infty}^{-\alpha+i\infty}\Gamma(s)\zeta^\beta(s)\frac{ds}{z^s} \right|^2 dx \ll_{\beta,\alpha} \frac{1}{\delta^{(\beta-1)(1+2\alpha)-1}}.\]
Now we choose $\alpha = \frac{\varepsilon}{2(\beta-1)}$ and establish the bound
\[J_5 \ll_{\beta,\varepsilon} \max \left\{\frac{1}{\delta^{\frac{\beta-1}{2}+\varepsilon}},\frac{1}{\delta^{\beta-2+\varepsilon}}\right\}, \quad \delta \to 0^+, \forall \varepsilon > 0, \forall \beta \geqslant 2.\]
Replacing $\delta$ by $\delta/2$ in the integral $\int_0^{+\infty}\abs{\Phi_\beta \left (xe^{i\left (\frac{\pi}{2} - \delta \right )} \right )}^2dx$, we obtain also the same estimation for $J_\beta(\delta)$ and this completes the proof of our theorem.
\section*{Acknowledgement}
\noindent I thank Prof. B. Conrey for many helpful discussions.

\end{document}